\documentclass{amsart}
\usepackage{amscd,amssymb}

\newcounter{noindnum}[subsection]
\renewcommand{\thenoindnum}{\alph{noindnum}}
\newcommand{\noindstep}{\refstepcounter{noindnum}{{\rm(}\thenoindnum}\/{\rm)} }
\newcommand{\stepzero}{\setcounter{noindnum}{0}}
\renewcommand{\phi}{\varphi}
\renewcommand{\epsilon}{\varepsilon}
\renewcommand{\emptyset}{\varnothing}

\newcommand{\bE}{\mathbf E}
\newcommand{\bG}{\mathbf G}
\newcommand{\bH}{\mathbf H}

\newcommand{\bT}{\mathbf T}
\newcommand{\bU}{\mathbf U}
\newcommand{\bu}{\mathbf u}
\newcommand{\bB}{\mathbf B}

\newcommand{\cA}{\mathcal A}
\newcommand{\cB}{\mathcal B}
\newcommand{\cE}{\mathcal E}
\newcommand{\cF}{\mathcal F}
\newcommand{\cG}{\mathcal G}
\newcommand{\cH}{\mathcal H}

\newcommand{\cO}{\mathcal O}

\newcommand{\cT}{\mathcal T}

\newcommand{\C}{\mathbb C}

\renewcommand{\P}{\mathbb P}
\newcommand{\A}{\mathbb A}

\DeclareMathOperator{\Rad}{Rad}
\DeclareMathOperator{\Gl}{Gl}

\DeclareMathOperator{\Id}{Id}

\DeclareMathOperator{\Aut}{Aut}
\DeclareMathOperator{\Iso}{Iso}

\DeclareMathOperator{\spec}{Spec}

\DeclareMathOperator{\Ker}{Ker}

\newcommand{\fm}{\mathfrak m}

\theoremstyle{plain}
\newtheorem{theorem}{Theorem}
\newtheorem*{conjecture}{Conjecture}
\newtheorem{proposition}{Proposition}[section]
\newtheorem{lemma}[proposition]{Lemma}
\newtheorem*{corollary*}{Corollary}
\newtheorem{corollary}[proposition]{Corollary}

\theoremstyle{definition}
\newtheorem{definition}[proposition]{Definition}

\theoremstyle{remark}

\newtheorem*{remarks*}{Remarks}
\newtheorem*{remark*}{Remark}

\newtheorem{construction}[proposition]{Construction}

\begin{document}

\title[Proof of Grothendieck--Serre conjecture]{Proof of Grothendieck-Serre conjecture
on principal bundles over regular local rings containing a finite field}

\keywords{Reductive group schemes; Principal bundles\newline
The author acknowledges support of the
RNF-grant 14-11-00456.}

\begin{abstract}
Let $R$ be a regular local ring, containing {\bf a finite field}. Let $\bG$ be a reductive group scheme over~$R$. We prove that a principal $\bG$-bundle over~$R$ is trivial, if it is trivial over the fraction field of $R$. In other words, if $K$ is the fraction field of $R$, then the map of non-abelian cohomology pointed sets
\[
  H^1_{\text{\'et}}(R,\bG)\to H^1_{\text{\'et}}(K,\bG),
\]
induced by the inclusion of $R$ into $K$, has a trivial kernel.

Certain arguments used in the present preprint do not work if the ring $R$ contains
a characteristic zero field.
In that case and, more generally, in the case when the regular local ring $R$ contains {\bf an infinite field} this result is proved in \cite{FP}.

\end{abstract}


\author{Ivan Panin}
\email{paniniv@gmail.com}
\address{Steklov Institute of Mathematics at St.-Petersburg, Fontanka 27, St.-Petersburg 191023, Russia}

\maketitle

\section{Introduction}

Assume that $U$ is a regular scheme, $\bG$ is a reductive $U$-group scheme. Recall that a $U$-scheme $\cG$ with an action of $\bG$ is called \emph{a principal $\bG$-bundle over $U$}, if $\cG$ is faithfully flat and quasi-compact over $U$ and the action is simple transitive, that is, the natural morphism $\bG\times_U\cG\to\cG\times_U\cG$ is an isomorphism, see~\cite[Section~6]{FGA}. It is well known that such a bundle is trivial locally in \'etale topology but in general not in Zariski topology. Grothendieck and Serre conjectured that $\cG$ is trivial locally in Zariski topology, if it is trivial generically. More precisely
\begin{conjecture}
Let $R$ be a regular local ring, let $K$ be its field of fractions. Let $\bG$ be a reductive group scheme over $U:=\spec R$, let $\cG$ be a principal $\bG$-bundle. If $\cG$ is trivial over $\spec K$, then it is trivial. Equivalently, the map of non-abelian cohomology pointed sets
\[
  H^1_{\text{\'et}}(R,\bG)\to H^1_{\text{\'et}}(K,\bG),
\]
induced by the inclusion of $R$ into $K$, has a trivial kernel.
\end{conjecture}
The main result of this paper is a proof of this conjecture for regular semi-local domains $R$, containing {\bf a finite field}.
Our proof was inspired by the preprint
[FP], where the conjecture is proven for semi-local regular domains containing an {\bf infinite} field.
{\it Thus,
the conjecture holds for
semi-local regular domains containing a field}.

The proof in the present preprint uses
\cite[Thm.1.1]{Pan1},
\cite[Thm.1.0.1]{Pan2},
the key ideas of the paper
\cite{FP}
and a Bertini type theorem from
\cite{Poo}.

Our result implies that two principal $\bG$-bundles over $U$ are isomorphic, if they are isomorphic over $\spec K$ as proved in the next section. This result is new even for constant group schemes (that is, for group schemes coming from the ground field).

Recall that a part of the Gersten conjecture asserts that the natural homomorphism of $\mathrm K$-groups $\mathrm K_q(R)\to\mathrm K_q(K)$ is injective. Very roughly speaking, the Grothendieck--Serre conjecture is a non-abelian version of this part of the Gersten conjecture.

\subsection{History of the topic}
Here is a list of known results in the same vein, corroborating the Grothendieck--Serre conjecture.

\smallskip $\bullet$ The case, where the group scheme $\bG$ comes from an infinite ground field, is completely solved by J.-L.~Colliot-Th\'el\`ene, M.~Ojanguren, and M.~S.~Raghunatan in~\cite{C-TO} and \cite{R1,R2}; O.~Gabber announced a proof for group schemes coming from arbitrary ground fields.

\smallskip $\bullet$ The case of an arbitrary reductive group scheme over a discrete valuation ring or over a henselian ring is completely solved by Y.~Nisnevich in~\cite{Ni1}. He also proved the conjecture for two-dimensional local rings in the case, when $\bG$ is quasi-split in~\cite{Ni2}.

\smallskip $\bullet$ The case, where $\bG$ is an arbitrary reductive group scheme over a regular semi-local domain containing an infinite field,
was settled by R.~Fedorov and I.~Panin in~\cite{FP}.

\smallskip $\bullet$ The case, where $\bG$ is an arbitrary torus over a regular local ring, was settled by J.-L.~Colliot-Th\'{e}l\`{e}ne and J.-J.~Sansuc in~\cite{C-T-S}.

\smallskip $\bullet$ For some simple group schemes of classical series the conjecture is solved in works of the author, A.~Suslin, M.~Ojanguren, and K.~Zainoulline; see~\cite{Oj1}, \cite{Oj2}, \cite{PS}, \cite{OP2}, \cite{Z}, \cite{OPZ}.

\smallskip $\bullet$ Under an isotropy condition on $\bG$ and assuming that the ring contains an infinite field the conjecture is proved in a series of preprints~\cite{PSV} and~\cite{Pa2}.

\smallskip $\bullet$ The case of strongly inner simple adjoint group schemes of the types $E_6$ and $E_7$ is done by the second author,
V.~Petrov, and A.~Stavrova in~\cite{PPS}. No isotropy condition is imposed there, however it is supposed that the ring contains an infinite field.

\smallskip $\bullet$ The case, when $\bG$ is of the type $F_4$ with trivial $g_3$-invariant and the field is of characteristic zero, is settled by V.~Chernousov in~\cite{Chernous}; the case, when $\bG$ is of the type $F_4$ with trivial $f_3$-invariant and the field is infinite and perfect, is settled by V.~Petrov and A.~Stavrova in~\cite{PetrovStavrova}.

\subsection{Acknowledgments}


The author thanks A.Suslin for his interest to the topic of the present preprint.

\section{Main results}\label{Introduction}
Let $R$ be a commutative unital ring. Recall that an $R$-group scheme $\bG$ is called \emph{reductive},
if it is affine and smooth as an $R$-scheme and if, moreover,
for each algebraically closed field $\Omega$ and for each ring homomorphism $R\to\Omega$ the scalar extension $\bG_\Omega$ is
a connected reductive algebraic group over $\Omega$. This definition of a reductive $R$-group scheme
coincides with~\cite[Exp.~XIX, Definition~2.7]{SGA3-3}.
A~well-known conjecture due to J.-P.~Serre and A.~Grothendieck
(see~\cite[Remarque, p.31]{Se}, \cite[Remarque 3, p.26-27]{Gr1}, and~\cite[Remarque~1.11.a]{Gr2})
asserts that given a regular local ring $R$ and its field of fractions~$K$ and given a reductive group scheme $\bG$ over $R$, the map
\[
  H^1_{\text{\'et}}(R,\bG)\to H^1_{\text{\'et}}(K,\bG),
\]
induced by the inclusion of $R$ into $K$, has a trivial kernel. The following theorem, which is the main result of the present paper,
asserts that this conjecture holds,
provided that $R$ contains {\bf a finite field}.
If $R$ contains an infinite field, then the conjecture is proved in [FP].

\begin{theorem}\label{MainThm1}
Let $R$ be a regular semi-local domain containing {\bf a finite field}, and let $K$ be its field of fractions. Let $\bG$ be
a reductive group scheme over $R$. Then the map
\[
  H^1_{\text{\'et}}(R,\bG)\to H^1_{\text{\'et}}(K,\bG),
\]
\noindent
induced by the inclusion of $R$ into $K$, has a trivial kernel. In other words, under the above assumptions on $R$ and $\bG$, each principal $\bG$-bundle over $R$ having a $K$-rational point is trivial.
\end{theorem}

Theorem~\ref{MainThm1} has the following
\begin{corollary*}
Under the hypothesis of Theorem~\ref{MainThm1}, the map
\[
  H^1_{\text{\'et}}(R,\bG)\to H^1_{\text{\'et}}(K,\bG),
\]
\noindent
induced by the inclusion of $R$ into $K$, is injective. Equivalently, if $\cG_1$ and $\cG_2$ are two principal bundles isomorphic over $\spec K$, then they are isomorphic.
\end{corollary*}
\begin{proof}
Let $\cG_1$ and $\cG_2$ be two principal $\bG$-bundles isomorphic over $\spec K$. Let $\Iso(\cG_1,\cG_2)$ be the scheme of isomorphisms. This scheme is a principal $\Aut\cG_2$-bundle. By Theorem~\ref{MainThm1} it is trivial, and we see that $\cG_1\cong\cG_2$.
\end{proof}

Note that, while Theorem~\ref{MainThm1} was previously known for reductive group schemes $\bG$ coming from the ground field
(an unpublished result due to O.Gabber), in many cases the corollary is a new result even for such group schemes.

For a scheme $U$ we denote by $\A^1_U$ the affine line over $U$ and by $\P^1_U$ the projective line over $U$. Let $T$ be a $U$-scheme. By a principal $\bG$-bundle over $T$ we understand a principal $\bG\times_UT$-bundle.

In Section~\ref{sect:redtopsv} we deduce Theorem~\ref{MainThm1} from the following result of independent interest (cf.~\cite[Thm.1.3]{PSV}).

\begin{theorem}\label{th:psv}
Let $R$ be the semi-local ring of finitely many closed points on an irreducible smooth affine variety over {\bf a finite field} $k$,
set $U=\spec R$. Let $\bG$ be a simple simply-connected group scheme over $U$ (see~\cite[Exp.~XXIV, Sect.~5.3]{SGA3-3} for the definition).
Let $\cE_t$ be a principal $\bG$-bundle over the affine line $\A^1_U=\spec R[t]$, and let $h(t)\in R[t]$ be a monic polynomial.
Denote by $(\A^1_U)_h$ the open subscheme in $\A^1_U$ given by $h(t)\ne0$ and assume that
the restriction of $\cE_t$ to $(\A^1_U)_h$ is a trivial principal $\bG$-bundle.
Then for each section $s:U\to\A^1_U$ of the projection $\A^1_U\to U$ the $\bG$-bundle $s^*\cE_t$ over $U$ is trivial.
\end{theorem}

The derivation of Theorem~\ref{MainThm1} from Theorem~\ref{th:psv} is based on
\cite[Thm.1.0.1]{Pan2} and~\cite[Thm.1.1]{Pan1}.

Let $Y$ be a semi-local scheme. We will call a simple $Y$-group scheme
quasi-split if its restriction to each connected component of $Y$ contains {\bf a Borel subgroup scheme}.

\begin{theorem}\label{MainThm2}
Let $R$, $U$, and $\bG$ be as in Theorem~\ref{th:psv}.
Let $Z\subset\P^1_U$ be a closed subscheme finite over $U$.
Let $Y\subset\P^1_U$ be a closed subscheme finite and
\'etale over $U$ and such that\\
(i) $\bG_Y:=\bG\times_UY$ is quasi-split, \\
(ii) $Y\cap Z=\emptyset$ and $Y \cap \{\infty\}\times U=\emptyset= Z \cap \{\infty\}\times U$, \\
(iii) for any closed point $u \in U$ one has $Pic(\P^1_u - Y_u)=0$, where $Y_u:=\P^1_u\cap Y$.\\
Let $\cG$ be a~principal $\bG$-bundle over
$\P^1_U$ such that its restriction to
$\P^1_U- Z$ is trivial.
Then the restriction of $\cG$ to
$\P^1_U-Y$ is also trivial.\\
In particular, the principal $\bG$-bundle $\cG$ is trivial locally for the Zarisky topology.
\end{theorem}

The proof of this result is inspired by \cite[Thm.3]{FP}.




\subsection{Organization of the paper}
In Section~\ref{sect:redtopsv}, we reduce Theorem~\ref{MainThm1} to Theorem~\ref{th:psv}. In Section~\ref{sect:reducing},
we reduce Theorem~\ref{th:psv} to Theorem~\ref{MainThm2}. This reduction is based
on~\cite[Thm.1.0.1]{Pan2}, \cite[Thm.1.1]{Pan1},
on a theorem of D.~Popescu~\cite{P}
and on
Proposition \ref{SchemeY}.
The latter proposition is a new ingredient
comparing with respecting arguments from
\cite[Section 4]{FP}.

In Section~\ref{sect:proof2} we prove Theorem~\ref{MainThm2}.
We give an outline of the proof in Section~\ref{sect:outline}.
We use the technique of henselization.


In Section~\ref{sect:application} we give an application of Theorem~\ref{MainThm1}.

In the Appendix we recall the definition of henselization from~\cite[Section~0]{Gabber}.

\section{Reducing Theorem~\ref{MainThm1} to Theorem~\ref{th:psv}}\label{sect:redtopsv}
In what follows ``$\bG$-bundle'' always means ``principal $\bG$-bundle''.  Now we assume that Theorem~\ref{th:psv} holds. We start with the following particular case of Theorem~\ref{MainThm1}.

\begin{proposition}\label{pr:geometric}
Let $R$, $U=\spec R$, and $\bG$ be as in Theorem~\ref{th:psv}. Let $\cE$ be a principal $\bG$-bundle over $U$, trivial at the generic point of $U$. Then $\cE$ is trivial.
\end{proposition}

\begin{proof}
Under the hypothesis of the proposition, the following data are constructed in
\cite[Thm.1.1]{Pan1}:\\
\stepzero\noindstep a principal $\bG$-bundle $\cE_t$ over $\A^1_U$;\\
\noindstep a monic polynomial $h(t)\in R[t]$.\\
Moreover these data satisfies the following conditions:\\
(1) the restriction of $\cE_t$ to $(\A^1_U)_h$ is a trivial principal $\bG$-bundle;\\
(2) there is a section $s:U\to\A^1_U$ such that $s^*\cE_t=\cE$.

Now it follows from Theorem~\ref{th:psv} that $\cE$ is trivial.
\end{proof}

\begin{proposition}\label{pr:reductivegeometric}
Let $U$ be as in Theorem~\ref{th:psv}. Let $\bG$ be a reductive group scheme over $U$.
Let $\cE$ be a principal $\bG$-bundle over $U$ trivial at the generic point of $U$. Then $\cE$ is trivial.
\end{proposition}

\begin{proof}
Firstly, using~\cite[Thm.1.0.1]{Pan2}, we can assume that~$\bG$ is semi-simple and simply-connected. Secondly, standard arguments (see for instance~\cite[Section~9]{PSV}) show that we can assume that $\bG$ is simple and simply-connected. (Note that for this reduction it is necessary to work with semi-local rings.) Now the proposition is reduced to Proposition~\ref{pr:geometric}.
\end{proof}

\begin{proof}[Proof of Theorem~\ref{MainThm1}]
Let us prove a general statement first. Let $k'$ be {\bf a finite field}, $X$ be a $k'$-smooth irreducible affine variety, $\bH$ be a reductive group scheme over $X$. Denote by $k'[X]$ the ring of regular functions on $X$ and by $k'(X)$ the field of rational functions on $X$. Let $\cH$ be a principal $\bH$-bundle over $X$ trivial over $k'(X)$. Let $\mathfrak p_1,\dots,\mathfrak p_n$ be prime ideals in $k'[X]$, and let $\cO_{\mathfrak p_1,\dots,\mathfrak p_n}$ be the corresponding semi-local ring.
\begin{lemma}\label{lm:primemax}
The principal $\bH$-bundle $\cH$ is trivial over $\cO_{\mathfrak p_1,\dots,\mathfrak p_n}$.
\end{lemma}
\begin{proof}
For each $i=1,2,\ldots,n$ choose a maximal ideal $\mathfrak m_i\subset k'[X]$ containing $\mathfrak p_i$. One has inclusions of $k'$-algebras
\[
\cO_{\mathfrak m_1,\dots,\mathfrak m_n}\subset\cO_{\mathfrak p_1,\dots,\mathfrak p_n}\subset k'(X).
\]
By Proposition~\ref{pr:reductivegeometric} the principal $\bH$-bundle $\cH$ is trivial over $\cO_{\mathfrak m_1,\dots,\mathfrak m_n}$. Thus it is trivial over $\cO_{\mathfrak p_1,\dots,\mathfrak p_n}$.
\end{proof}

Let us return to our situation. Let $\fm_1,\ldots,\fm_n$ be all the maximal ideals of $R$.
Let $\cE$ be a $\bG$-bundle over $R$ trivial over the fraction field of $R$.
Clearly, there is a non-zero $f\in R$ such that $\cE$ is trivial over $R_f$.
Let $k$ be the prime field of $R$.
Note that $k$ is perfect. It follows from Popescu's theorem (\cite{P,Sw}) that~$R$ is a filtered inductive limit of smooth $k$-algebras $R_\alpha$.
Modifying the inductive system $R_\alpha$ if necessary, we can assume that each $R_\alpha$ is integral.
There are an index $\alpha$, a reductive group scheme $\bG_{\alpha}$ over $R_{\alpha}$,
a principal $\bG_{\alpha}$-bundle $\cE_{\alpha}$ over $R_{\alpha}$,
and an element $f_{\alpha }\in R_{\alpha}$ such that
$\bG=\bG_{\alpha}\times_{\spec R_{\alpha}}\spec R$, $\cE$ is isomorphic to
$\cE_{\alpha}\times_{\spec R_{\alpha}}\spec R$ as principal $\bG$-bundle,
$f$ is the image of $f_{\alpha}$ under the homomorphism
$\phi_{\alpha}: R_{\alpha}\to R$, $\cE_{\alpha}$ is trivial over $(R_{\alpha})_{f_{\alpha}}$.

For each maximal ideal $\mathfrak m_i$ in~$R$ ($i=1,\dots, n$) set $\mathfrak p_i=\phi_{\alpha}^{-1}(\mathfrak m_i)$.
The homomorphism $\phi_\alpha$ induces a homomorphism of semi-local rings $(R_{\alpha})_{\mathfrak p_1,\dots,\mathfrak p_n}\to R$.
By Lemma~\ref{lm:primemax} the principal $\bG_{\alpha}$-bundle $\cE_{\alpha}$ is trivial over
$(R_{\alpha})_{\mathfrak p_1,\dots,\mathfrak p_n}$. Whence the $\bG$-bundle $\cE$ is trivial over $R$.
\end{proof}


\section{Reducing Theorem~\ref{th:psv} to Theorem~\ref{MainThm2}}\label{sect:reducing}
Now we assume that Theorem~\ref{MainThm2} is true. Let $k$, $U$ and $\bG$ be as in Theorem~\ref{th:psv}.
Let $u_1,\ldots,u_n$ be all the closed points of $U$. Let $k(u_i)$ be the residue field of $u_i$. Consider the reduced closed subscheme $\bu$ of $U$,
whose points are $u_1$, \ldots, $u_n$. Thus
\[
 \bu\cong\coprod_i\spec k(u_i).
\]
Set $\bG_\bu=\bG\times_U\bu$. By $\bG_{u_i}$ we denote the fiber of $\bG$ over $u_i$;
it is a simple simply-connected algebraic group over $k(u_i)$.

\begin{proposition}
\label{SchemeY}
Let $Z\subset\A^1_U$ be a closed subscheme finite over $U$.
There is a closed subscheme $Y\subset\A^1_U$ which is \'etale and finite over $U$ and such that \\
(i) $\bG_Y:=\bG\times_UY$ is quasi-split, \\
(ii) $Y\cap Z=\emptyset$,\\
(iii) for any closed point $u \in U$ one has $Pic(\P^1_u - Y_u)=0$, where $Y_u:=\P^1_u\cap Y$.\\
(Note that $Y$ and $Z$ are closed in $\P^1_U$ since they are finite over $U$).
\end{proposition}

\begin{proof}
For every $u_i$ in $\bu$ choose a Borel subgroup $\bB_{u_i}$ in $\bG_{u_i}$.
{\it The laller is possible since the fields $k(u_i)$ are finite.}
Let $\cB$ be the $U$-scheme of Borel subgroup schemes of $\bG$.
It is a smooth projective $U$-scheme (see~\cite[Cor.~3.5, Exp.~XXVI]{SGA3-3}).
The subgroup $\bB_{u_i}$ in $\bG_{u_i}$ is a $k(u_i)$-rational point~$b_i$ in the fibre of $\cB$ over the point $u_i$.
Using a variant of Bertini theorem
({\bf see \cite[Thm.1.2]{Poo}}),
we can find a closed subscheme $Y^{\prime}$ of $\cB$ such that
$Y^{\prime}$ is \'etale over $U$ and all the $b_i$'s are in $Y$
(take an embedding of $\cB$ into a projective space $\P^N_U$ and
intersect $\cB$  with appropriately chosen family of hypersurfaces containing the points $b_i$.
Arguing as in the proof of \cite[Lemma~7.2]{OP2},
we get a scheme $Y^{\prime}$ finite and \'etale over $U$).
For any closed point $u_i$ in $U$ the fibre $Y^{\prime}_{u_i}$ of $Y^{\prime}$ over $u_i$
contains a $k(u_i)$-rational point (it is the point $b_i$).

To continue the proof of the Proposition we need the following
\begin{lemma}
\label{F1F2}
Let $U$ be as in the Proposition.
Let $Z\subset\A^1_U$ be a closed subscheme finite over $U$.
Let $Y^{\prime} \to U$ be a finite \'{e}tale morphism such that
for any closed point $u_i$ in $U$ the fibre $Y^{\prime}_{u_i}$ of $Y^{\prime}$ over $u_i$
contains a $k(u_i)$-rational point. Then there are finite field extensions
$k_1$ and $k_2$ of the finite field $k$ such that \\
(i) the degrees $[k_1: k]$ and $[k_2: k]$ are coprime,\\
(ii) $k(u_i) \otimes_k k_r$ is a field for $r=1$ and $r=2$,\\
(iii) the degrees $[k_1: k]$ and $[k_2: k]$ are strictly greater than any of the degrees
$[k(z): k]$ , where $z$ runs over all closed points of $Z$,\\
(iv) there is a closed embedding of $U$-schemes
$Y^{\prime\prime}=((Y^{\prime}\otimes_k k_1) \coprod (Y^{\prime}\otimes_k k_2)) \xrightarrow{i} \A^1_U$,\\
(v) for $Y=i(Y^{\prime\prime})$ one has $Y \cap Z = \emptyset$,\\
(vi) for any closed point $u_i$ in $U$ one has
$Pic(\P^1_{u_i}-Y_{u_i})=0$.
\end{lemma}
To prove this Lemma note that it's easy to find field extensions $k_1$ and $k_2$
subjecting (i) to (iii). To satisfy (iv) it suffices to require that
for any closed point $u_i$ in $U$ and for $r=1$ and $r=2$
the number of closed points in
$Y^{\prime}_{u_i}\otimes_k k_r$
is the same as the number of closed points in
$Y^{\prime}_{u_i}$,
and to require that for any integer $n>0$ and any closed point $u_i$ in $U$
the number of points
$y \in Y^{\prime\prime}_{u_i}$
with
$[k(y): k(u_i)]=n$
is not more than the number of points
$x \in \A^1_{u_i}$
with
$[k(x): k(u_i)]=n$.
Clearly, these requirements can be satisfied,
which proves the item (iv).

The condition (v) holds for any closed $U$-embedding
$i: Y^{\prime\prime} \hookrightarrow \A^1_U$ from item (iv),
since the property (iii).
The condition (vi) holds since the property (i).

Now complete the proof of Proposition \ref{SchemeY}. Take the $U$-scheme
$Y^{\prime} \subset \bB$ as in the beginning of the proof.
This $U$-scheme $Y^{\prime}$ satisfies the assumption of Lemma
\ref{F1F2}. Take the closed subscheme $Y$ of $\A^1_U$ as in the item (v)
of the Lemma. For this $Y$ the conditions (ii) and (iii) of the Proposition are
obviously satisfied. The condition (i) is satisfied too, since already
it is satisfied for the $U$-scheme $Y^{\prime}$. The Proposition follows.

\end{proof}

\begin{proof}[Proof of Theorem~\ref{th:psv}]
Set $Z:=\{h=0\}\cup s(U)\subset\A^1_U$.
Clearly, $Z$ is finite over $U$.
Since the principal $\bG$-bundle $\cE_t$ is trivial over $(\A^1_U)_h$
it is trivial over $\A^1_U-Z$.
Note that $\{h=0\}$ is closed in $\P^1_U$ and finite over $U$ because $h$ is monic.
Further, $s(U)$ is also closed in $\P^1_U$ and finite over $U$ because it is a zero set of a degree one monic polynomial.
Thus $Z\subset\P^1_U$ is closed and finite over $U$.


Since the principal $\bG$-bundle $\cE_t$ is trivial over $(\A^1_U)_h$, and
$\bG$-bundles can be glued in Zariski topology,
there exists a principal $\bG$-bundle $\cG$ over $\P^1_U$ such that\\
\indent (i) its restriction to $\A^1_U$ coincides with $\cE_t$;\\
\indent (ii) its restriction to $\P^1_U-Z$ is trivial.

Now choose $Y$ in $\A^1_U$ as in Proposition
~\ref{SchemeY}.
Clearly, $Y$ is finite \'{e}tale over $U$ and closed in $\P^1_U$.
Moreover,
$Y \cap \{\infty\}\times U=\emptyset= Z \cap \{\infty\}\times U$
and
$Y\cap Z= \emptyset$.
Applying Theorem~\ref{MainThm2} with this choice of $Y$ and $Z$,
we see that the restriction of $\cG$ to $\P^1_U-Y$ is a trivial $\bG$-bundle.
Since $s(U)$ is in $\A^1_U-Y$
and $\cG|_{\A^1_U}$ coincides with $\cE_t$,
we conclude that $s^*\cE_t$ is a trivial principal $\bG$-bundle over $U$.
\end{proof}

\section{Proof of Theorem~\ref{MainThm2}}\label{sect:proof2}
We will be using notation from Theorem~\ref{MainThm2}.
Let $\bu$
be as in Section~\ref{sect:reducing}. For $u\in\bu$ set $\bG_u=\bG|_u$.

\begin{proposition}\label{pr:trivclsdfbr}
Let $\cE$ be a $\bG$-bundle over $\P^1_U$ such that $\cE|_{\P^1_u}$ is a trivial $\bG_u$-bundle for all $u\in\bu$.
Assume that there exists a closed subscheme $T$ of $\P^1_U$ finite over $U$ such that the restriction of $\cE$ to $\P^1_U-T$ is trivial
and $(\infty \times U) \cap T = \emptyset$. Then $\cE$ is trivial.
\end{proposition}
\begin{proof}
This follows from Theorem~9.6 of~\cite{PSV}, since $\cE|_{(\infty \times U)}$ is a trivial $\bG$-bundle.
\end{proof}

\subsection{An outline of the
proof of Theorem~\ref{MainThm2}}\label{sect:outline}
Our proof of this Theorem almost literally coincides with the proof of
\cite[Thm.3]{FP}. Our arguments are simpler at certain points.

An outline of the proof.

Denote by $Y^h$ the henselization of the pair $(\A_U^1,Y)$, it is a scheme over $\A_U^1$. Let $s:Y\to Y^h$ be the canonical closed embedding, see Section~\ref{sect:distinguishedLimit} for more details. Set $\dot Y^h:=Y^h-s(Y)$. Let $\cG'$ be a $\bG$-bundle over $\P^1_U-Y$. Denote by $\Gl(\cG',\phi)$ the $\bG$-bundle over $\P_U^1$ obtained by gluing $\cG'$ with the trivial $\bG$-bundle $\bG\times_U Y^h$ via a $\bG$-bundle isomorphism $\phi:\bG\times_U\dot Y^h\to\cG'|_{\dot Y^h}$.


Note that the $\bG$-bundle $\cG$ can be presented in the form $\Gl(\cG',\phi)$, where $\cG'= \cG|_{\P^1_U-Y}$. The idea is to show that

\vskip4pt
\newlength{\shortwidth}
\setlength{\shortwidth}{\textwidth}
\addtolength{\shortwidth}{-.5cm}

\noindent ($*$) \emph{\hfill \parbox{\shortwidth}{There is $\alpha\in\bG(\dot Y^h)$ such that
the $\bG_\bu$-bundle $\Gl(\cG',\phi\circ\alpha)|_{\P^1_\bu}$ is trivial (here $\alpha$ is regarded as
an automorphism of the $\bG$-bundle $\bG\times_U\dot Y^h$ given by the right translation by the element $\alpha$).}}

\vskip4pt

If we find $\alpha$ satisfying condition ($*$), then Proposition~\ref{pr:trivclsdfbr}, applied to $T=Y\cup Z$,
shows that the $\bG$-bundle $\Gl(\cG',\phi\circ\alpha)$ is trivial over $\P^1_U$.
On the other hand, its restriction to $\P^1_U-Y$ coincides with the $\bG$-bundle $\cG'=\cG|_{\P^1_U-Y}$.
\emph{Thus $\cG|_{\P^1_U-Y}$ is a trivial $\bG$-bundle}.

To prove ($*$) it suffices to show that\\
\indent (i) the bundle $\cG|_{\P^1_\bu-Y_\bu}$ is trivial;\\
\indent (ii) each element $\gamma_\bu\in\bG_\bu(\dot Y_\bu^h)$ can be written in the form
\[
 \alpha|_{\dot Y_\bu^h}\cdot\beta_\bu|_{\dot Y_\bu^h}
\]
for certain elements $\alpha\in\bG(\dot Y^h)$ and $\beta_\bu\in\bG_\bu(Y_\bu^h)$.

A realization of this plan in details is given below in the paper.

\subsection{Henselization of affine pairs}\label{sect:distinguishedLimit}
We will use the theory of henselian pairs and, in particular,
a notion of a henselization $A^h_I$ of a commutative ring $A$ at an ideal $I$ (see~Appendix and~\cite[Section~0]{Gabber}).
We refer to
\cite[subsection 5.2]{FP}
for the geometric counterpart.
Let $S=\spec A$ be a scheme and $T=\spec(A/I)$ be a closed subscheme.
Let $(T^h,\pi: T^h \to S,s: T \to T^h)$ be
\emph{the henselization of the pair $(S,T)$} (cf. Definition~\ref{def:henzelisation}).
By definition the scheme $T^h$ is {\it affine} and the composite morphism
$\pi \circ s: T \to S$ is the closed embedding $T \hookrightarrow S$.
{\it Recall} that the pair $(T^h,s(T))$ is henselian, which means that for any affine \'etale morphism $\pi:Z\to T^h$,
any section $\sigma$ of $\pi$ over $s(T)$ uniquely extends to a section of $\pi$ over $T^h$.
It is known that $\pi^{-1}(T)=s(T)$.

In the notation of~\cite[Section~0]{Gabber} we have $T^h=\spec A^h_I$, $\pi:T^h\to S$ is induced by the structure of $A$-algebra on $A^h_I$.

{\it Recall three properties of henselization of affine pairs}\\
\indent (i) Let $T$ be a semi-local scheme. Then the henselization commutes with restriction to closed subschemes.
In more details, if $S'\subset S$ is a closed subscheme,
then
there is a natural morphism
$(T\times_SS')^h\to T^h\times_SS'$.
This morphism is an isomorphism and the canonical section $s':T\times_SS'\to(T\times_SS')^h$ coincides under this identification with
\[
    s\times_S\Id_{S'}:T\times_SS'\to T^h\times_SS'.
\]
%

\indent (ii) If $T=\coprod_i T_i$ is a disjoint union, then $T^h=\coprod_i T_i^h$.\\
\indent (iii) If we replace in a pair $(S,T)$ the scheme $S$ by an \'{e}tale affine neighborhood of $T$, then the
$(T^h,\pi,s)$ {\it remains the same}. In more details,
given a pair $(S,T)$ as above we write temporarily
$(S^{hen}_T, \pi_{S,T}, s_{S,T})$ for $(T^h,\pi,s)$.
If $p: W \to S$ is an \'{e}tale morphism and $t: T \hookrightarrow W$ is such that $p \circ t: T \hookrightarrow S$ coincides
with the closed embedding $T$ into $S$,
then there is a canonical {\it isomorphism}
$\rho: W^{hen}_T \to S^{hen}_T$
of the $S$-schemes
$(W^{hen}_T, \pi_{W,T})$ and $(S^{hen}_T, \pi_{S,T})$
such that
$\rho \circ s_{W,T}=s_{S,T}$.


\subsection{Gluing principal $\bG$-bundles}
Recall that $U=\spec R$, where $R$ is the semi-local ring of finitely many closed points on an irreducible $k$-smooth affine variety over a finite field $k$.
Also, $\bG$ is a simple simply-connected group scheme over $U$, and
$Y$ is a closed subscheme of $\P_U^1$ finite and \'etale over $U$.

{\it We
will assume below in the preprint that} $Y\subset\A_U^1$
(as in the hypotheses of Theorem \ref{MainThm2}).
Let $(Y^h,\pi,s)$ be the henselization of the pair $(\A_U^1,Y)$
and let
$\dot Y^h=Y^h-s(Y)$ and let $in: \A_U^1 \hookrightarrow \P_U^1$ be the open inclusion.
\begin{proposition}\cite{FP}
\label{pr:affine}\stepzero
The schemes $Y^h$ and $\dot Y^h$ are affine.
\end{proposition}

Let us make a general remark. Let $\cF$ be a $\bG$-bundle over a $U$-scheme $T$. By definition, a trivialization of $\cF$ is a $\bG$-equivariant isomorphism $\bG\times_UT\to\cF$. Equivalently, it is a section of the projection $\cF\to T$. If $\phi$ is such a trivialization and $f:T'\to T$ is a $U$-morphism, we get a trivialization $f^*\phi$ of $f^*\cF$. Sometimes we denote this trivialization by $\phi|_{T'}$. We also sometimes call a trivialization of $f^*\cF$ \emph{a trivialization of $\cF$ on $T'$}.





The main cartesian square we will work with is
\begin{equation}
\begin{CD}\label{eq:distinguishedLimit}
\dot Y^h @>>> Y^h\\
@VVV @VV{in\circ\pi}V\\
\P^1_U - Y @>>>\P^1_U.
\end{CD}
\end{equation}

Let $\cA$ be  the category  of pairs $(\cE,\psi)$, where $\cE$ is a $\bG$-bundle on $\P_U^1$, $\psi$ is a trivialization of $\cE|_{Y^h}:=(in\circ\pi)^*\cE$. A morphism between $(\cE,\psi)$ and $(\cE',\psi')$ is an isomorphism $\cE\to\cE'$ compatible with trivializations.

Similarly, let $\cB$ be the category of pairs $(\cE,\psi)$, where $\cE$ is a $\bG$-bundle on $\P_U^1-Y$, $\psi$ is a trivialization of $\cE|_{\dot Y^h}$.


Consider the restriction functor $\Psi:\cA\to\cB$.
\begin{proposition}\cite{FP}
\label{pr:gluing2}
    The functor $\Psi$ is an equivalence of categories.
\end{proposition}

\begin{construction}\cite{FP}
\label{equally_well}
By Proposition~\ref{pr:gluing2} we can choose a functor quasi-inverse to $\Psi$. Fix such a functor $\Theta$. Let $\Lambda$ be the forgetful functor from $\cA$ to the category of $\bG$-bundles over $\P_U^1$. For $(\cE,\psi)\in\cB$ set
\[
    \Gl(\cE,\psi)=\Lambda(\Theta(\cE,\psi)).
\]
Note that $\Gl(\cE,\psi)$ comes with a canonical trivialization over $Y^h$.

Conversely, if $\cE$ is a principal $\bG$-bundle over $\P_U^1$ such that its restriction to $Y^h$ is trivial, then $\cE$ can be represented as $\Gl(\cE',\psi)$, where $\cE'=\cE|_{\P_U^1-Y}$, $\psi$ is a trivialization of $\cE'$ on $\dot Y^h$.
\end{construction}

Let $\bu$ be as in Section~\ref{sect:reducing}, $Y_\bu:=Y\times_U\bu$.
Let $(Y_\bu^h,\pi_\bu,s_\bu)$ be the henselization of $(\A^1_\bu,Y_\bu)$.
Using property~(i) of henselization, we get $Y_\bu^h=Y^h\times_U\bu$.
Thus we have a natural closed embedding $Y_\bu^h\to Y^h$.
Set $\dot Y_\bu^h=Y_\bu^h-s_\bu(Y_\bu)$.
We get a closed embedding
\begin{equation}
\label{dotYdotYh}
\dot Y_\bu^h\hookrightarrow \dot Y^h.
\end{equation}
Thus the pull-back of the cartesian square
(\ref{eq:distinguishedLimit})
by means of the closed embedding $\bu\hookrightarrow U$ has the form
\[
\begin{CD}
\dot Y_\bu^h @>>>Y_\bu^h\\
@VVV @VV in_\bu\circ\pi_\bu V\\
\P^1_\bu - Y_\bu @>>>\P^1_\bu,
\end{CD}
\]
where $in_\bu:\A_\bu^1\to\P_\bu^1$.

Similarly to the above, we can define categories $\cA_\bu$ and $\cB_\bu$ and an equivalence of categories
$\Psi_\bu:\cA_\bu\to\cB_\bu$.
Let $\Theta_\bu$ be a functor quasi-inverse to $\Psi_\bu$ and $\Lambda_\bu$ be the forgetful functor from $\cA_\bu$
to the category of $\bG_\bu$-bundles over $\P_\bu^1$.
Let $\cE_\bu$ be a principal $\bG_\bu$-bundle over $\P^1_\bu-Y_\bu$ and $\psi_\bu$ be a trivialization of $\bG_\bu$ on $\dot Y_\bu^h$.
Set
\[
\Gl_\bu(\cE_\bu,\psi_\bu)=\Lambda_\bu(\Theta_\bu(\cE_\bu,\psi_\bu)).
\]
\begin{lemma}\cite{FP}
\label{basechange_limit}
Let $(\cE,\psi)\in\cB$, and let $\Gl(\cE,\psi)$ be the $\bG$-bundle obtained by Construction~\ref{equally_well}. Then
\[
 \Gl_\bu(\cE|_{\P^1_\bu-Y_\bu},\psi|_{\dot Y_\bu^h}) \ \text{and} \ \Gl(\cE,\psi)|_{\P^1_\bu}
\]
are isomorphic as $\bG_\bu$-bundles over $\P^1_\bu$.
\end{lemma}

\begin{lemma}\cite{FP}
\label{coboundary_limit}
For any $(\cE_\bu,\psi_\bu)\in\cB_\bu$ and any $\beta_\bu\in\bG_\bu(Y_\bu^h)$ the $\bG_\bu$-bundles
\[
 \Gl_\bu(\cE_\bu,\psi_\bu)\ \text{and} \ \Gl_\bu(\cE_\bu,\psi_\bu\circ\beta_\bu|_{\dot Y_\bu^h})
\]
are isomorphic (here $\beta_\bu|_{\dot Y_\bu^h}$ is regarded as an automorphism of the $\bG_\bu$-bundle $\bG_\bu\times_\bu\dot Y_\bu^h$ given by the right translation by $\beta_\bu|_{\dot Y_\bu^h}$).
\end{lemma}


\subsection{Proof of Theorem~\ref{MainThm2}: presentation of $\cG$ in the form $\Gl(\cG',\phi)$}\label{sect:presentation}
Let $U$, $\bG$, $Z$, $Y$ and $\cG$ be as in Theorem~\ref{MainThm2}.

\begin{proposition}\cite{FP}
\label{presentation}
The $\bG$-bundle $\cG$ over $\P^1_U$ is of the form $\Gl(\cG',\phi)$ for the $\bG$-bundle $\cG':=\cG|_{\P^1_U-Y}$ and a
trivialization $\phi$ of $\cG'$ over $\dot Y^h$.
\end{proposition}
\begin{proof}
In view of Construction~\ref{equally_well},
it is enough to prove that the restriction of the principal $\bG$-bundle $\cG$ to $Y^h$ is trivial.
Let us choose a closed subscheme $Z'\subset\A^1_U$ such that $Z'$ contains $Z$, $Z'\cap Y=\emptyset$, and $\A^1_U-Z'$ is affine.
Then $\A^1_U-Z'$ is an affine neighborhood of $Y$. By the property (iii) from subsection
\ref{sect:distinguishedLimit}
the henselization of the pair $(\A^1_U-Z',Y)$ coincides with
the henselization of the pair $(\A^1_U,Y)$.
Since $\cG$ is trivial over $\A^1_U-Z'$, its pull-back to $Y^h$ is trivial too.
The proposition is proved.
\end{proof}

\emph{Our aim is to modify the trivialization $\phi$ via an element
\[
   \alpha\in\bG(\dot Y^h)
\]
so that the $\bG$-bundle $\Gl(\cG',\phi\circ\alpha)$ becomes trivial over} $\P^1_U$.

\subsection{Proof of Theorem~\ref{MainThm2}: proof of property~(i) from the outline}\label{sect:properties_i}
Now we are able to prove property (i) from the outline of the proof. In fact, we will prove the following
modification of
\cite[Lemma 5.11]{FP}.
\begin{lemma}\label{podpravka}
Let $\Gl(\cG',\phi)$ be the presentation of the $\bG$-bundle $\cG$ over~$\P^1_U$ given in
Proposition~\ref{presentation}.
Set $\phi_\bu:=\phi|_{\dot Y_\bu^h}$.
Then there is $\gamma_\bu\in\bG_\bu(\dot Y_\bu^h)$
such that the $\bG_\bu$-bundle
$\Gl_\bu(\cG'|_{\P^1_\bu-Y_\bu},\phi_\bu\circ\gamma_\bu)$
is trivial.
\end{lemma}

\begin{proof}
We show first that $\cG|_{\P^1_\bu-Y_\bu}$ is trivial. One has
\[
    \P^1_\bu= \coprod_{u\in\bu}\P^1_u
\]
For $u\in\bu$ set $Y_u:=Y\times_Uu$, $\bG_u:=\bG\times_Uu$, and $\cG_u:=\cG\times_Uu$.


Take $u\in\bu$. By our assumption on $Y$,
$Pic(\P^1_u-Y_u)=0$.
The $\bG_u$-bundle $\cG_u$ is trivial over $\A^1_u-Z_u$.
Thus, by~\cite[Corollary~3.10(a)]{GilleTorseurs},
it is trivial over $\P^1_u-Y_u$.

We see that
$\cG'|_{\P^1_\bu-Y_\bu}=\cG|_{\P^1_\bu-Y_\bu}$
is trivial.
Choosing a trivialization, we may identify $\phi_\bu$ with an element of
$\bG_\bu(\dot Y_\bu^h)$. Set $\gamma_\bu=\phi_\bu^{-1}$.
By the very choice of~$\gamma_\bu$ the $\bG_\bu$-bundle
$\Gl_\bu(\cG'|_{\P^1_\bu-Y_\bu},\phi_\bu\circ\gamma_\bu)$
is trivial.
\end{proof}


\subsection{Proof of Theorem~\ref{MainThm2}: reduction to property (ii) from the outline}\label{sect:properties_ii}
The aim of this section is to deduce Theorem~\ref{MainThm2} from the following
\begin{proposition}\cite{FP}
\label{alpha}
Each element $\gamma_\bu\in\bG_\bu(\dot Y_\bu^h)$ can be written in the form
\[
 \alpha|_{\dot Y_\bu^h}\cdot\beta_\bu|_{\dot Y_\bu^h}
\]
for certain elements $\alpha\in\bG(\dot Y^h)$ and $\beta_\bu\in\bG_\bu(Y_\bu^h)$.
\end{proposition}
\begin{proof}[Deduction of Theorem~\ref{MainThm2} from Proposition~\ref{alpha}]\cite{FP}
Let $\Gl(\cG',\phi)$ be the presentation of the $\bG$-bundle $\cG$ from Proposition~\ref{presentation}.
Let $\gamma_\bu\in\bG_\bu(\dot Y_\bu^h)$ be the element from Lemma~\ref{podpravka}.
Let $\alpha\in\bG(\dot Y^h)$ and $\beta_\bu\in\bG_\bu(Y_\bu^h)$ be the elements from Proposition~\ref{alpha}. Set
\[
 \cG^{new}=\Gl(\cG',\phi\circ\alpha).
\]
\emph{Claim.} The $\bG$-bundle $\cG^{new}$ is trivial over $\P^1_U$.\\ Indeed, by Lemmas~\ref{basechange_limit} and~\ref{coboundary_limit} one has a chain of isomorphisms of $\bG_\bu$-bundles
\begin{multline*}
  \cG^{new}|_{\P^1_\bu}\cong
  \Gl_\bu(\cG'|_{\P^1_\bu-Y_\bu},\phi_\bu \circ \alpha|_{\dot Y_\bu^h})
\cong\\
\Gl_\bu(\cG'|_{\P^1_\bu-Y_\bu},\phi_\bu \circ
\alpha|_{\dot Y_\bu^h}\circ\beta_\bu|_{\dot Y_\bu^h})=
\Gl_\bu(\cG'|_{\P^1_\bu-Y_\bu}, \phi_\bu \circ \gamma_\bu),
\end{multline*}
which is trivial by the choice of $\gamma_\bu$.
The $\bG$-bundles $\cG|_{\P^1_U-Y}$ and $\cG^{new}|_{\P^1_U-Y}$ coincide
by the very construction of~$\cG^{new}$.
By Proposition~\ref{pr:trivclsdfbr},
applied to $T=Z\cup Y$, the $\bG$-bundle $\cG^{new}$ is trivial.
Whence the claim.

The claim above implies that the $\bG$-bundle $\cG|_{\P^1_U-Y}= \cG^{new}|_{\P^1_U-Y}$ is trivial.
Theorem~\ref{MainThm2} is proved.
\end{proof}

\subsection{End of proof of Theorem~\ref{MainThm2}: proof of property (ii) from the outline}
\emph{In the remaining part of Section~\ref{sect:proof2} we will prove Proposition~\ref{alpha}. This will complete the proof of Theorem~\ref{MainThm2}}.

By our assumption on $Y$, the group scheme $\bG_Y=\bG\times_UY$ is {\bf quasi-split}.
Thus we can and will choose a Borel subgroup scheme $\bB^+$ in $\bG_Y$.

Since $Y$ is an affine scheme,
by~\cite[Exp.~XXVI, Cor.~2.3, Th~4.3.2(a)]{SGA3-3}
there is an opposite to $\bB^+$ Borel subgroup scheme $\bB^-$ in $\bG_Y$. Let $\bU^+$ be the unipotent radical of
$\bB^+$, and let $\bU^-$ be the unipotent radical of $\bB^-$.

\begin{definition}\label{EYi}
We will write $\bE$ for the functor, sending a $Y$-scheme $T$ to the subgroup $\bE(T)$ of the group $\bG_Y(T)=\bG(T)$ generated by the subgroups $\bU^+(T)$ and $\bU^-(T)$ of the group $\bG_Y(T)=\bG(T)$.
\end{definition}

\begin{lemma}
\label{lm:surjectivity}
The functor $\bE$ has the property that for every closed subscheme $S$ in an affine $Y$-scheme $T$ the induced map $\bE(T)\to\bE(S)$ is surjective.
\end{lemma}
\begin{proof}
The restriction maps $\bU^\pm(T)\to\bU^\pm(S)$ are surjective, since $\bU^\pm$ are isomorphic to vector bundles as $Y$-schemes
(see~\cite[Exp.~XXVI, Cor.~2.5]{SGA3-3}).
\end{proof}

Recall that $(Y^h,\pi,s)$ is the henselization of the pair $(\A_U^1,Y)$. Also, $in:\A_U^1\to\P_U^1$ is the embedding. Denote the projection $\A_U^1\to U$ by $pr$ and the projection $\A_Y^1\to Y$ by $pr_Y$.
\begin{lemma}\cite{FP}
\label{lm:retraction}
There is a morphism $r:Y^h\to Y$ making the following diagram commutative
\begin{equation}
\label{eq:retraction}
\begin{CD}
Y^h @>r>> Y\\
@V{in\circ\pi}VV @VV{pr|_{Y}}V\\
\P^1_U @>pr>> U
\end{CD}
\end{equation}
and such that $r\circ s=\Id_Y$.
\end{lemma}

We view $Y^h$ as a $Y$-scheme via $r$. Thus various subschemes of $Y^h$ also become $Y$-schemes. In particular, $\dot Y^h$ and $\dot Y_\bu^h$ are $Y$-schemes, and we can consider
\[
\bE(\dot Y^h)\subset\bG(\dot Y^h) \ \text{ and } \ \bE(\dot Y_\bu^h)\subset\bG(\dot Y_\bu^h)=\bG_\bu(\dot Y_\bu^h).
\]

\begin{lemma}\label{nastia}
\[
    \bG_\bu(\dot Y_\bu^h)=\bE(\dot Y_\bu^h)\bG_\bu(Y_\bu^h).
\]
\end{lemma}
\begin{proof}
Firstly, one has $Y_\bu=\coprod_{u\in\bu}\coprod_{y\in Y_u}y$. (Note that $Y_u$ is a finite scheme.) Thus by property (ii) of henselization, we have
\[
    Y_\bu^h=\coprod_{u\in\bu}\coprod_{y\in Y_u}y^h,\qquad
    \dot Y_\bu^h=\coprod_{u\in\bu}\coprod_{y\in Y_u}\dot y^h,
\]
where $(y^h,\pi_y,s_y)$ is the henselization of the pair
$(\A_u^1,y)$,
$\dot y^h:=y^h-s_y(y)$.
We see that $y^h$ and $\dot y^h$ are subschemes of $Y^h$, so we can view them as $Y$-schemes, and
$\bG_{y^h}:=\bG_Y\times_Y{y^h}$ is {\bf quasi-split}.
Also, $\bE(\dot y^h)$ makes sense as a subgroup of
$\bG(\dot y^h)=\bG_u(\dot y^h)=\bG_{y^h}(\dot y^h)$.

One has
\[
\begin{split}
   \bG_\bu(\dot Y_\bu^h)&=\prod_{u\in\bu}\prod_{y\in Y_u}\bG_u(\dot y^h)=\prod_{u\in\bu}\prod_{y\in Y_u}\bG_{y^h}(\dot y^h),\\
   \bE(\dot Y_\bu^h)&=\prod_{u\in\bu}\prod_{y\in Y_u}\bE(\dot y^h),\\
   \bG_\bu(Y_\bu^h)&=\prod_{u\in\bu}\prod_{y\in Y_u}\bG_u(y^h)=\prod_{u\in\bu}\prod_{y\in Y_u}\bG_{y^h}(y^h).
\end{split}
\]
Thus it suffices for each $u\in\bu$ and each $y\in Y_u$ to check the equality
\[
   \bG_{y^h}(\dot y^h)=\bE(\dot y^h)\bG_{y^h}(y^h).
\]
This equality holds by Fait 4.3 and Lemma 4.5 of~\cite{Gille:BourbakiTalk}.
In fact, $y^h=\spec\mathcal O$,
where $\mathcal O=k(u)[t]_{\mathfrak m_y}^h$ is a henselian discrete valuation ring,
and $\mathfrak m_y\subset k(u)[t]$ is the maximal ideal defining the point $y\in\A^1_u$.
Further, $\dot y^h=\spec L$, where $L$ is the fraction field of $\mathcal O$.
The lemma is proved.
\end{proof}

We have the closed embedding
(\ref{dotYdotYh})
and the scheme
$\dot Y^h$
is affine by Proposition~\ref{pr:affine}.
Recall that we regard
$\dot Y^h$ as a $Y$-scheme via the morphism
$r|_{\dot Y^h}$.
Thus
by Lemma~\ref{lm:surjectivity}
the restriction map $\bE(\dot Y^h)\to\bE(\dot Y_\bu^h)$ is surjective.
Since $\bE(\dot Y^h)\subset\bG(\dot Y^h)$, the proposition
\ref{alpha}
follows. \emph{This completes the proof of Theorem~\ref{MainThm2}}.

\section{An application}\label{sect:application}
The following result is a straightforward consequence of Theorem~\ref{MainThm1} and an exact sequence for \'etale cohomology.
Recall that by our definition a reductive group scheme has geometrically connected fibres.
\begin{theorem}\label{Norms}
Let $R$ be as in Theorem~\ref{MainThm1} and $\bG$ be a reductive $R$-group scheme. Let $\mu:\bG\to\bT$ be a group scheme morphism to an $R$-torus $\bT$ such that $\mu$ is locally in the \'{e}tale topology on $\spec R$ surjective. Assume further that the $R$-group scheme $\bH:=\Ker(\mu)$ is reductive. Let $K$ be the fraction field of $R$. Then the group homomorphism
\[
\bT(R)/\mu(\bG(R))\to\bT(K)/\mu(\bG(K))
\]
is injective.
\end{theorem}
This theorem extends all the known results of this form proven in~\cite{C-TO}, \cite{PS}, \cite{Z}, \cite{OPZ}.


\appendix
\section{\cite{FP}}
For a commutative ring $A$ we denote by $\Rad(A)$ its Jacobson ideal. The following definition one can find in~\cite[Section 0]{Gabber}.
\begin{definition}
If $I$ is an ideal in a commutative ring $A$, then the pair $(A,I)$ is called \emph{henselian\/}, if $I\subset\Rad(A)$ and for every two relatively prime monic polynomials $\bar g, \bar h\in\bar A[t]$, where $\bar A=A/I$, and monic lifting $f\in A[t]$ of $\bar g\bar h$, there exist monic liftings $g,h\in A[t]$ such that $f=gh$. (Two polynomials are called relatively prime, if they generate the unit ideal.)
\end{definition}

\begin{lemma}\cite{FP}
\label{basechhens}
Let $(A,I)$ be a henselian pair with a semi-local ring $A$ and $J\subset A$ be an ideal. Then the pair $(A/J,(I+J)/J)$ is henselian.
\end{lemma}



The following definition one can find in~\cite[Section 0]{Gabber}.
\begin{definition}\label{def:henzelisation}
The henselization of any pair $(A,I)$ is the pair $(A_I^h,I^h)$ (over $(A,I)$) defined as follows
\[
(A_I^h,I^h):=\text{the filtered inductive limit over the category $\mathcal N$ of }(A',\Ker(\sigma)),
\]
where $\mathcal N$ is the filtered category of pairs $(A',\sigma)$ such that $A'$ is an \'{e}tale $A$-algebra and $\sigma\in\text{Hom}_{A-alg}(A',A/I)$.
\end{definition}


\end{document}